\documentclass[11pt,hidelinks]{article}

\title{Gaussian deconvolution and the lace expansion \\
for spread-out models}

\author{
Yucheng Liu\thanks{Department of Mathematics,
	University of British Columbia,
	Vancouver, BC, Canada V6T 1Z2.
	Liu: \url{https://orcid.org/0000-0002-1917-8330},
		\href{mailto:yliu135@math.ubc.ca}{yliu135@math.ubc.ca}.
	Slade: \url{https://orcid.org/0000-0001-9389-9497},
		\href{mailto:slade@math.ubc.ca}{slade@math.ubc.ca}.
	}
\and
Gordon Slade$^*$
}
\date{\vspace{-5ex}} 

\usepackage{amsmath, amssymb, amscd, amsthm, amsfonts}
\usepackage{graphicx} 
\usepackage{hyperref} 
\usepackage{booktabs}
\usepackage{xcolor}
\usepackage{ dsfont, bm}
\usepackage[title]{appendix}
\usepackage{comment}
\usepackage{enumerate}
\usepackage{enumitem}
\usepackage{cite}

\usepackage[textwidth=480pt,textheight=650pt,centering]{geometry} 

\theoremstyle{plain}
\newtheorem{theorem}{Theorem}[section]
\newtheorem{lemma}[theorem]{Lemma}

\newtheorem{proposition}[theorem]{Proposition}

\newtheorem{definition}[theorem]{Definition}
\newtheorem{assumption}[theorem]{Assumption}

\theoremstyle{definition}
\newtheorem{remark}[theorem]{Remark}

\numberwithin{equation}{section}

\newcommand{\eps}{\varepsilon}

\newcommand{\Z}{\mathbb{Z}}

\newcommand{\R}{\mathbb{R}}
\newcommand{\C}{\mathbb{C}}

\newcommand{\T}{\mathbb{T}}

\newcommand{\grad}{\nabla}
\newcommand{\inv}{^{-1}}
\renewcommand{\(}{\left(}
\renewcommand{\)}{\right)}
\newcommand{\half}{\frac{1}{2}}

\newcommand{\1}{\mathds{1}}

\newcommand{\nl}{\nonumber \\}

\providecommand{\abs}[1]{\lvert#1\rvert}
\providecommand{\norm}[1]{\lVert#1\rVert}
\providecommand{\Bigabs}[1]{\Big\lvert#1\Big\rvert}
\providecommand{\Bignorm}[1]{\Big\lVert#1\Big\rVert}
\providecommand{\bignorm}[1]{\big\lVert#1\big\rVert}

\usepackage{mathrsfs}

\newcommand{\nnb}{\nonumber\\}

\newcommand{\veee}[1]{|\!|\!|#1|\!|\!|}
\newcommand{\xvee}{\veee{x}}
\providecommand{\nnnorm}[1]{\veee #1}
\newcommand{\const}{\mathrm{const}}

\newcommand{\Dnn}{P}

\newcommand{\lambdaSO}{\lambda}
\newcommand{\muSO}{\mu}
\newcommand{\crit}{_{z_c}}
\newcommand{\muz}{{\muSO_z}}
\newcommand{\HH}{G}
\newcommand{\G}{  \tilde G }

\begin{document}
\maketitle

\begin{abstract}
We present a new proof of $\abs x^{-(d-2)}$ decay of
critical two-point functions for spread-out statistical mechanical models on $\Z^d$ above
the upper critical dimension,
based on the lace expansion and assuming appropriate diagrammatic estimates.
Applications include spread-out models of
the Ising model and
self-avoiding walk in dimensions
$d >4$, and spread-out percolation for $d>6$.
The proof is based on an extension of the
new Gaussian deconvolution theorem we obtained in a recent paper. It
provides a technically simpler and conceptually more transparent approach than the
method of Hara, van der Hofstad and Slade (2003).
\end{abstract}

%


\section{Introduction and results}
\label{sec:so}

\subsection{Introduction}

The lace expansion has been used to prove mean-field behaviour for several statistical
mechanical models above their upper critical dimension $d_c$, including self-avoiding walk
($d_c=4$),
the Ising model ($d_c=4$), percolation
($d_c=6$), and lattice trees and lattice animals ($d_c=8$).
The expansion requires a small parameter for its
convergence, which roughly speaking
is $(d-d_c)\inv$ for nearest-neighbour models, so $d$ is required to be large,
e.g., $d \ge 11$ for percolation \cite{FH17}.  (An exception is self-avoiding walk
for which dimensions $d \ge 5$ have been handled with computer assistance \cite{HS92a}.)
This obscures the role of the upper
critical dimension, and in order to apply lace expansion methods to analyse
critical behaviour in all dimensions $d>d_c$, spread-out models were first used
in \cite{HS90a}.  Spread-out models extend nearest-neighbour connections to include
long
connections (finite-range or rapidly decaying) parametrised by a large
parameter $L\gg 1$.
The reciprocal of $L$ provides a small parameter
that can be used to obtain convergence of the lace expansion without the need to
take the dimension artificially large.
In addition, the proof of mean-field
behaviour for a wide class of spread-out models serves as a demonstration of universality.

In \cite{HHS03,Saka07}, $|x|^{-(d-2)}$ decay of critical two-point functions
was proved for spread-out versions of all the above-mentioned models above $d_c$
(i.e., the critical exponent $\eta$ is equal to zero), and this decay has important
applications, e.g., \cite{HJ04,HMS23,KN11} for percolation.
The proofs in \cite{HHS03,Saka07} are based on a bootstrap argument and a deconvolution theorem developed in
\cite{HHS03}.  The method of \cite{HHS03} involves intricate Fourier analysis,
and our purpose in this paper is to provide a technically simpler and conceptually
clearer replacement for the method of \cite{HHS03}.  Our method is based on
the deconvolution theorem of \cite{LS24a}, which itself was inspired by \cite{Slad22_lace}.
The deconvolution theorem of \cite{LS24a} provides a new and simpler proof of
Hara's Gaussian Lemma \cite{Hara08}, using only elementary facts about the Fourier transform, product and quotient rules of differentiation, and H\"older's inequality.
We adapt it here to apply to spread-out models.
In the process, we rely on results proved in \cite{LS24a}, particularly those
concerning the $L^p$ theory of the Fourier transform in the context of weak derivatives.

The adaptation requires control of the
$L$-dependence in estimates, which feeds into the bootstrap argument
that is central to the convergence proof for the lace expansion.
The general deconvolution theorem of \cite{LS24a} does not, on its own, provide
sufficient control in error bounds to effectively deal with this interrelation.
For this reason, we also develop and present a generic approach to the bootstrap analysis
in the context of spread-out models.

For self-avoiding walk, the lace expansion
\cite{BS85,Slad06} produces a
convolution equation for the two-point function $G$, of the form
\begin{equation}
\label{eq:FG}
    F*G=\delta
\end{equation}
with $(F*G)(x) = \sum_{y\in \Z^d} F(x-y)G(y)$, $\delta$ the Kronecker delta $\delta(x) = \delta_{0,x}$, and $F$ explicitly defined by the lace expansion.
We refer to \eqref{eq:FG} as the \emph{impulse} equation.
Our deconvolution theorem gives conditions on $F$ which guarantee  $G(x)$ to have $|x|^{-(d-2)}$ decay at criticality.
For percolation or the Ising model, the lace expansion
\cite{HS90a, Saka07, HH17book}
instead produces an inhomogeneous convolution equation
of the form
\begin{equation}
\label{eq:Hh}
    \tilde{F}*\HH = h  ,
\end{equation}
with explicit functions $\tilde F$ and $h$.
We will prove that in the spread-out setting, \eqref{eq:Hh} can be reduced to
the simpler \eqref{eq:FG},
so that our deconvolution theorem then extends to \eqref{eq:Hh}.
In \cite{HHS03}, \eqref{eq:FG} is instead rewritten in the form \eqref{eq:Hh} in order
to handle both equations simultaneously.  Our reduction is more efficient,
as it reduces the more difficult equation to the simpler one, rather than vice versa.

\medskip\noindent
{\bf Notation.}  We write
$a \vee b = \max \{ a , b \}$ and $a \wedge b = \min \{ a , b \}$.
We write
$f = O(g)$ or $f \lesssim g$ to mean there exists a constant $C> 0$ such that $\abs {f(x)} \le C\abs {g(x)}$, and $f= o(g)$ for $\lim f/g = 0$.
To avoid dividing by zero, with $|x|$ the Euclidean norm of $x\in \R^d$, we define
\begin{equation}
\label{eq:xvee1}
	\xvee = \max\{|x| , 1\}.
\end{equation}
Note that \eqref{eq:xvee1} does not define a norm on $\R^d$.

\medskip\noindent
{\bf Fourier transform.}
Let $\T^d=(\R/2\pi\Z)^d$ denote the continuum torus, which
we identify with $(-\pi,\pi]^d \subset \R^d$.
For a function $f \in L^1(\Z^d)$, the Fourier transform and its inverse are given by
\begin{equation}
\hat f(k)  = \sum_{x\in\Z^d}f(x) e^{ik\cdot x} 	\quad (k \in \T^d),
\qquad
f(x) = \int_{\T^d} \hat f(k) e^{-ik\cdot x}  \frac{ dk }{ (2\pi)^d } \quad (x\in \Z^d).
\end{equation}

\subsection{Main results}
\label{sec:so-def}

\subsubsection{Spread-out random walk}

Our point of reference is the Green function (two-point function) of
a spread-out random walk on $\Z^d$, whose
transition probability $D: \Z^d \to [0, 1]$
is given by the following definition.

\begin{definition}
\label{def:D}
Let $v : \R^d \to [0, \infty)$ be bounded,
$\Z^d$-symmetric (invariant under reflection in coordinate hyperplanes
and rotation by $\pi/2$), supported in $[-1, 1]^{d}$,
with $\frac{\partial^d v}{\partial x_1 \cdots \partial x_d}$
piecewise continuous,
and with $\int_{[-1, 1]^{d}}v(x) d x = 1$.
Given $v$, we define $D: \Z^d \to [0, 1]$ by
\begin{equation}
\label{eq:Ddef}
    D(x) = \frac{v(x/L)}{\sum_{x\in \Z^d}v(x/L)}.
\end{equation}
\end{definition}

By definition, $D$ is always supported in $[-L, L]^d$.
For example, if $v(x) = 2^{-d} \1 \{ \norm x_\infty \le 1 \}$, then
$D$ is the uniform distribution on $[-L, L]^d \cap \Z^d $.

\smallskip\noindent
{\bf Large $L$ assumption.}
For the above definition to make sense, the denominator in
\eqref{eq:Ddef} must be nonzero.
It will be nonzero if $L$ is large enough (depending on $d$ and $v$),
since $\sum_{x\in \Z^d}v(x/L) \sim L^{d}$ as $L\to\infty$
by a Riemann sum approximation to $\int_{[-1, 1]^{d}}v(x)dx$.
Similarly, the variance $\sigma^2 = \sum_{x\in\Z^d} |x|^2 D(x)$
of $D$ as is asymptotic to a multiple of $L^2$.
We assume throughout the entire paper that
$L \ge L_0$ with $L_0$ chosen large enough that the denominator of \eqref{eq:Ddef} is
nonzero.
Moreover, when required we will increase the value of $L_0$, but only in a manner
that depends on $d$ and $v$ alone.
Since we work throughout with a fixed dimension $d$
and a fixed function $v$,  we do not track the dependence
on $d$ or $v$ of constants in bounds:  all constants in bounds are permitted to depend on $d$
and $v$.

\medskip
For $\mu \in [0,1]$, we let $S_\mu$ denote the \emph{spread-out Green function},
which is the solution of the convolution equation
\begin{equation} \label{eq:S}
    (\delta - \mu D)*S_\mu = \delta
\end{equation}
given by
\begin{equation}
S_\mu(x)= \int_{\T^d}\frac{e^{-ik\cdot x}}{1-\mu\hat D(k)} \frac{dk}{(2\pi)^d}
\qquad (d>2).
\end{equation}
When $D(x)$ is replaced by $\Dnn(x) = \frac{1}{2d} \1_{|x|=1}$, we have the nearest-neighbour random walk, and we denote its Green
function by $C_\mu(x)$.
When $\mu\in (0,1)$, both $S_\mu$ and $C_\mu$ decay exponentially.
In the critical case $\mu=1$, it is well-known (e.g., \cite{LL10, Uchi98}) that
\begin{equation}
\label{eq:C1_asymp}
    C_1(x) =
    \frac { a_d }{\xvee^{d-2} }
    + O\(\frac 1 {\xvee^d }\) ,
    \qquad a_d = \frac{ d \Gamma(\frac{ d-2 } 2 ) }{ 2\pi^{d/2}}
    \qquad (d>2).
\end{equation}
For $S_1(x)$, 	we establish a similar statement in Proposition~\ref{prop:so-nn}.

\begin{proposition}[Spread-out Green function]
\label{prop:so-nn}
Let $d>2, \eps >0$, and $L \ge L_0$. 
Then
\begin{align}
\label{eq:CL-uni2}
    S_1(x) & =  \delta_{0,x} + \frac{1}{\sigma^2} C_1(x)
    +O\( \frac{1}{L^{1-\eps} \xvee^{d-1}} \),
\end{align}
with the constant uniform in $L$ but dependent on $\eps$.
Also, there is a constant $K_S = K_S(\eps)$ such that
\begin{equation}
\label{eq:S1bd}
    S_1(x) \leq  \delta_{0,x}+ K_S L^{-(2-\eps)} \xvee^{-(d-2)} \qquad (x \in \Z^d).
\end{equation}
\end{proposition}

Proposition~\ref{prop:so-nn} shows that $S_1(x)$ has the same $\abs x^{-(d-2)}$ decay as $C_1(x)$.
When the constant in the error term is permitted to depend on $L$,
Proposition~\ref{prop:so-nn} is a standard result and is
known to hold with error $O(\xvee^{-d})$ like \eqref{eq:C1_asymp}, e.g., \cite{LL10,Uchi98}.
However, the uniformity in $L$ needs care, is not standard, and is required for our results.
It is proved in
\cite[Proposition~1.6]{HHS03}, using intricate Fourier analysis, that
\begin{align}
\label{eq:CL-HHS}
    S_1(x) & =  \frac{a_{d}}{\sigma^2} \frac{1}{\xvee^{d-2}}
    +O\( \frac{1}{\xvee^{d-2+s}}
    \) ,
\end{align}
for any $s < 2$, with the constant independent of $L$.
For the case $s=1$, our error estimate in \eqref{eq:CL-uni2} has a
good power of $L$ compared to \eqref{eq:CL-HHS}.
We prove Proposition~\ref{prop:so-nn} using the simple strategy introduced in
\cite{LS24a}, as an alternative to the proof in \cite{HHS03}.
Via the fractional derivative analysis of \cite[Section~2.3]{LS24a}, the error term in \eqref{eq:CL-uni2} can in fact be improved to
$O( L^{-(2-t)} \xvee^{-(d-2+s)})$
with any $0 \le s <t <2$ and with the constant independent of $L$.
We omit the details of this improvement, which we do not need or invoke later.

\subsubsection{Spread-out Gaussian deconvolution theorem}

We now turn to general spread-out models.
The following assumption isolates basic properties that are typical of the
two-point function of models such as self-avoiding walk or percolation.
Our restriction to $z_c \ge 1$ is a convenience
that can be achieved by scaling
the model's parameter (e.g., the bond occupation probability $p$ for percolation).

\begin{assumption}
\label{ass:G}
Let $d>2$.  We assume that $G_z:\Z^d \to [0,\infty]$,
defined for $z \in [1, z_c]$ with some $z_c \ge 1$,
is a family of $\Z^d$-symmetric functions such that:
\begin{enumerate}[label=(\roman*)]
\item
$\sum_{x\in\Z^d} G\crit(x) = \infty$,

\item
for each $z < z_c$, $\sum_{x\in\Z^d} G_z(x) < \infty$ and
$G_z(x) = o(\abs x^{-(d-2)})$ as $\abs x \to \infty$ (need not be uniform in $z$),
\item
for each $x$, $G_z(x)$ is non-decreasing and continuous in $z \in [1,z_c]$,

\item
$G_1(x) \le S_1(x)$ for all $x\in \Z^d$.
\end{enumerate}
\end{assumption}

It is not part of the assumption that the critical $G\crit(x)$ is finite,
and the goal is to prove that $G\crit(x)$ has Gaussian $\abs x^{-(d-2)}$ decay.
To do this, we first establish a uniform in $z < z_c$ bound using the model-independent bootstrap argument of \cite{HHS03}, and then we bound $G\crit$ using monotone convergence.
The bootstrap argument compares $G_z$ to an upper bound of $S_1$, as follows.
Let $d>2$, fix a small $\eps>0$,
and let $K_S = K_S(\eps)$ be the constant in the bound \eqref{eq:S1bd}.
Given $L$, we define the \emph{bootstrap function} $b:[1,z_c]\to [0,\infty]$ by
\begin{equation}
    b(z) = \max \Big\{
    \sup_{x\neq 0} \frac{G_z(x)}{ K_S L^{-2+\eps}|x|^{-(d-2)} } ,\,
    3(z-1) \Big\} .
\end{equation}
The function $b(z)$ is finite and continuous in $z\in [1,z_c)$
by Assumption~\ref{ass:G}(ii)--(iii),
and $b(1) \le 1$ by Assumption~\ref{ass:G}(iv).

The next assumption gives consequences of an \emph{a priori} bound $b(z) \le 3$.
It reflects the fact that when $b(z) \le 3$,
diagrams arising from the lace expansion, which are functionals of $G_z$, can be bounded by functionals of the explicit function $L^{-2+\eps}|x|^{-(d-2)}$.
The verification of Assumption~\ref{ass:diagram}
is model dependent and is carried out
for spread-out models of self-avoiding walk, Ising model, percolation, and lattice
trees and lattice animals, above their upper critical dimensions, in \cite{HHS03,Saka07}; its verification requires large $L$.
We do not verify Assumption~\ref{ass:diagram} here.
We will later prove that $b(z) \le 2$
for all $z\le z_c$ via the bootstrap argument,
again assuming large $L$.
The specific $L$-dependence in the upper bound of \eqref{eq:Pibd} plays an
important role in the bootstrap, and the
proof of $b(z) \le 2$
requires that dependence to mesh well with the $L$-dependence in the upper bound \eqref{eq:S1bd}
on $S_1(x)$.  This is a subtlety that does not occur in \cite{LS24a} but that must
be dealt with for spread-out models.

\begin{assumption}[Lace expansion]
\label{ass:diagram}
Let $d\ge 1$, let $D$ be given by Definition~\ref{def:D}, and let $\rho > \frac{d-8}2 \vee 0$.
If $b(z) \le 3$ then there exists a $\Z^d$-symmetric function $\Pi_z : \Z^d \to \R$
for which
\begin{equation} \label{eq:FGz_boot}
    G_z = \delta + zD*G_z + \Pi_z*G_z,
\end{equation}
and
\begin{equation}
\label{eq:Pibd}
    \abs{ \Pi_z (x) }
    \le
    K_\Pi L^{-2+\epsilon} \Big(  \delta_{0,x} +
    \frac{ o(1)} {\xvee^{d+2+\rho}} \Big)
	\qquad	(x \in \Z^d),
\end{equation}
where the constant
$K_\Pi$ is independent of $z$ and $L$, and $o(1)\to 0$ uniformly in $z$
and $x$ as $L \to \infty$.
\end{assumption}

Equation \eqref{eq:FGz_boot} is what the lace expansion produces for self-avoiding walk, and it
can be rewritten as the impulse equation
\begin{equation}
\label{eq:Glace}
    (\delta - zD - \Pi_z)*G_z = \delta .
\end{equation}
For percolation
and the Ising model, the lace expansion instead produces an inhomogeneous convolution
equation
\begin{equation}
\label{eq:Hhconv}
    \HH_z = h_z + zD * h_z * \HH_z
    \quad\iff\quad  (\delta - zD * h_z)* \HH_z = h_z,
\end{equation}
with an explicit function $h_z$.
The next assumption covers these models with inhomogeneous convolution equation.

\begin{assumption}[Inhomogeneous lace expansion]
\label{ass:diagram-h}
Let $d\ge 1$, let $D$ be given by Definition~\ref{def:D} and let $\rho > \frac{d-8}2 \vee 0$.
If $b(z) \le 3$ then
there exists a $\Z^d$-symmetric function $h_z : \Z^d \to \R$
for which
\begin{equation} \label{eq:Hhlace}
    \HH_z = h_z + zD*h_z*\HH_z,
\end{equation}
and
\begin{align}
\label{eq:h_decay-1}
    \abs{ h_z(x) - \delta_{0,x} }
    \le K_h L^{-2+\epsilon}\Big( \delta_{0,x} + \frac{ o(1) } { \xvee^{d+2+\rho}  }\Big)
    	\qquad	(x \in \Z^d),
\end{align}
where the constant
$K_h$ is independent of $z$ and $L$, and $o(1)\to 0$ uniformly in $z$
and $x$ as $L \to \infty$.
\end{assumption}

\begin{proposition}
\label{prop:general-1}
Let $d \ge 1$.
Suppose $h_z: \Z^d \to \R$ is $\Z^d$-symmetric and
satisfies \eqref{eq:h_decay-1} and its uniformity assumptions with $\rho > -2$.
Then there is an $L_1$ (depending on $K_h$, $\eps$, and the $o(1)$)
such that, for all $L \ge L_1$,
there exists a $\Z^d$-symmetric function $\Phi_z : \Z^d \to \R$
for which $\HH_z$ obeying \eqref{eq:Hhlace}
satisfies $F_z * \HH_z = \delta$ with
\begin{align} \label{eq:h-1}
F_z = \delta - zD - \Phi_z,
\qquad
    \abs{ \Phi_z(x) }
    \le
    K_\Phi L^{-2+\eps} \Big( \delta_{0,x} + \frac{ o(1) } { \xvee^{d+2+\rho}  } \Big)
    \qquad	 (x \in \Z^d),
\end{align}
with the constant
$K_\Phi$ independent of $z$ and $L$, and $o(1)\to 0$ uniformly in $z$
and $x$ as $L \to \infty$.
\end{proposition}

In brief,
Proposition~\ref{prop:general-1} shows that Assumption~\ref{ass:diagram-h} implies Assumption~\ref{ass:diagram}, with $\Phi_z$ playing the role of $\Pi_z$:
the inhomogeneous equation \eqref{eq:Hhconv} can be rewritten
in the form of the impulse equation \eqref{eq:Glace}
when $h_z$ obeys \eqref{eq:h_decay-1}.
The following theorem is our main result.
A minor additional assumption allows its hypothesis $d>4$ to be weakened to $d>2$;
see Remark~\ref{rk:d2}.

\begin{theorem}
\label{thm:Gcrit}
Let $d>4$.
Under Assumption~\ref{ass:G}, together with either Assumption~\ref{ass:diagram} or
Assumption~\ref{ass:diagram-h},
there is an $L_2$ such that for all $L\ge L_2$,
\begin{align} \label{eq:Gcrit}
G\crit(x) \sim
\frac{ \lambdaSO\crit } { \sigma^2  } \frac {a_d }{\abs x^{d-2}}
\qquad \text{as $\abs x\to \infty$},
\end{align}
where $a_d$ is the constant of \eqref{eq:C1_asymp},
and, for any fixed $\eps>0$, $\lambdaSO\crit = 1 + O(L^{-2+\eps})$.
\end{theorem}

The constant $\lambda_{z_c}$ is
given explicitly in \eqref{eq:lambda-critical} in terms of $\Pi_{z_c}$.
Theorem~\ref{thm:Gcrit} is essentially proved in \cite{HHS03}, with an explicit error term.
Our contribution here is
to provide a different proof based on the simple
strategy of \cite{LS24a},
which replaces the more difficult and less conceptual analysis in \cite[p.358, items~(i) and (iv)]{HHS03}.

It is part of the proof of Theorem~\ref{thm:Gcrit}
that the bootstrap function obeys $b(z_c) \le 2$,
so that Assumption~\ref{ass:diagram} immediately verifies the hypotheses of \cite[Theorem~1.2]{LS24a}. That theorem then immediately
improves \eqref{eq:Gcrit} with an explicit error term:
\begin{align} \label{eq:Gcrit-error}
G\crit(x) =
\frac{ \lambdaSO\crit } { \sigma^2  } \frac {a_d }{\xvee^{d-2}}
+  \frac{O_L(1)}{\nnnorm x^{d-2+s}} ,
\end{align}
with any $s < \rho \wedge 2 \wedge (\rho - \frac{d-8}2)$.
In applications, there is generally
a sufficiently large but fixed value of $L$, and $L$-dependence of the error term
in \eqref{eq:Gcrit-error} is of no importance.
On the other hand, \cite[Theorem~1.2]{LS24a} cannot be applied {\it ab initio}, because
some control of $L$-dependence of the error term is needed in order to complete the bootstrap
argument used to prove Theorem~\ref{thm:Gcrit}.  This delicate point will
materialise below, in Theorem~\ref{thm:so} and in the proof of
Proposition~\ref{prop:boot}.

The verification of Assumption~\ref{ass:diagram}
or \ref{ass:diagram-h} is model-specific and requires $d > d_c$, where
\begin{equation}
    d_c =
    \begin{cases}
    4 & \text{(self-avoiding walk, Ising)}
    \\
    6 & \text{(percolation)}
    \\
    8 & \text{(lattice trees/animals)}.
    \end{cases}
\end{equation}
For each of these models,
Assumption~\ref{ass:diagram} or \ref{ass:diagram-h}
has been verified for large $L$
in \cite{HHS03, Saka07} with
\begin{equation}
\label{eq:rho-models}
    \rho =
    \begin{cases}
    2(d-4) & \text{(self-avoiding walk, Ising)}
    \\
    d-6 & \text{(percolation)}
    \\
    d-8 & \text{(lattice trees/animals)}.
    \end{cases}
\end{equation}
The requirement that $\rho > \frac{ d-8 } 2 \vee 0$ is satisfied in all cases.

\subsection{Organisation}

In Section~\ref{sec:proof-boot}, we prove our main result Theorem~\ref{thm:Gcrit}
subject to a general Gaussian deconvolution theorem (Theorem~\ref{thm:so}),
Proposition~\ref{prop:so-nn}, and Proposition~\ref{prop:general-1}. 
The proof of Theorem~\ref{thm:so} is given in
Section~\ref{sec:proof-so}, and the proof of Proposition~\ref{prop:so-nn} is given
in Appendix~\ref{sec:so-Green}. Both proofs use the simple strategy of \cite{LS24a}.
In Section~\ref{sec:general}, we introduce a novel way to reduce the
inhomogeneous lace expansion convolution equation \eqref{eq:Hhconv} for percolation, Ising model and lattice trees/animals,
to the self-avoiding walk equation~\eqref{eq:Glace}, thereby proving Proposition~\ref{prop:general-1}.
Finally, Appendix~\ref{app:D} provides a proof of some useful properties of
the transition probability $D$.

\section{Proof of Theorem~\ref{thm:Gcrit}}
\label{sec:proof-boot}

In Section~\ref{sec:so-deconv}, we state a general Gaussian deconvolution theorem (Theorem~\ref{thm:so}) in the spread-out setting.
The theorem involves an interval of $z$ values $[1,\infty)$
and functions $F_z:\Z^d \to \R$, and its statement concerns
the large-$x$ behaviour of the Fourier integral
\begin{equation} \label{eq:Gzint}
    \G_z(x) = \int_{\T^d} \frac{ e^{-ik\cdot x} }{\hat F_z(k)} \frac{dk}{(2\pi)^d}.
\end{equation}
Our assumption on $F_z$ is motivated by \eqref{eq:Pibd} and \eqref{eq:Glace}.
In Section~\ref{sec:boot},
we prove our main result Theorem~\ref{thm:Gcrit} using Theorem~\ref{thm:so} and a generic bootstrap argument.  Part of the proof involves verifying that under Assumptions~\ref{ass:G}
and \ref{ass:diagram},
$G_z$ is indeed equal to the Fourier integral $\G_z$.

\subsection{Gaussian deconvolution}
\label{sec:so-deconv}

The assumption on $F_z$ is the following.
It involves generic parameters $\beta_0,\beta_1$ in place of the
specific $L$-dependent choices made in \eqref{eq:Pibd}.
The \emph{a priori} bound $b(z)\le 3$
in Assumption~\ref{ass:diagram} is absent here,
but in Section~\ref{sec:boot} it will reappear.

\begin{assumption}
\label{ass:Fso}
Suppose that $D$ is given by Definition~\ref{def:D} 
and that $z \ge 1$.
We assume $F_z$ is given by $F_z = \delta - zD - \Pi_z$, where
$\Pi_z$ is a $\Z^d$-symmetric function that satisfies
\begin{align}
\label{eq:Pi_decay}
\abs{ \Pi_z(x) } \le \beta_0 \delta_{0,x} + \frac{ \beta_1 } { \xvee^{d+2+\rho}  }
\end{align}
with some $\rho > \frac{d-8}{2} \vee 0$ and $\beta = \beta_0 \vee \beta_1 \ge 0$ sufficiently small.
Suppose also that $\hat F_z(0) \ge 0$.
\end{assumption}

As we will see in \eqref{eq:IR-so}, Assumption~\ref{ass:Fso} implies an \emph{infrared bound} for $F_z$, namely that
there is a constant $K_{\rm IR}>0$ for which
\begin{align}
\label{eq:FIR}
\hat F_z(k) - \hat F_z(0) \ge K_{\rm IR} (L^2 \abs k^2 \wedge 1)  \qquad (k \in \T^d).
\end{align}
In dimensions $d >2$, the infrared bound implies absolute convergence of the Fourier integral
\eqref{eq:Gzint}.

Define
\begin{align} \label{eq:lambda_z-so}
\lambdaSO_z = \frac 1 { \hat F_z(0) - \sigma^{-2}\sum_{x\in\Z^d} \abs x^2 F_z(x) },	
\qquad \muSO_z = 1 - \lambdaSO_z \hat F_z(0).
\end{align}
By Assumption~\ref{ass:Fso}, we have $\hat F_z(0) =1 - z - \hat \Pi_z(0) \ge 0$.  Since also $z \ge 1$,
this implies that
\begin{equation}
\label{eq:zbd}
    1 \le z \le  1 -\hat\Pi_z(0) \le 1+O(\beta).
\end{equation}
It follows that
$0 \le \hat F_z(0) \le  O(\beta)$, and hence
\begin{align}
\label{eq:lambda1}
    \lambdaSO_z
    & = \frac{1}{ \hat F_z(0) + z + \sigma^{-2}\sum_{x\in\Z^d} |x|^2 \Pi_z(x)}
    =    \frac{1}{1+O(\beta) + O(\beta\sigma^{-2})}
    = 1+ O(\beta).
\end{align}
Since $0 \le \lambdaSO_z\hat F_z(0) \le O(\beta)$, we have $\muSO_z \in [1-O(\beta),1]$, so it makes sense to write $S_{\muSO_z}$ (recall \eqref{eq:S}).
We also have $\muSO_z \ge 1/2$ because $\beta$ is small.
Given $\rho > \frac{d-8}{2} \vee 0$, we define
\begin{align} \label{eq:nddef}
n_d = \begin{cases}
    d-2 & (\rho \le 1 + (\frac{d-8}{2} \vee 0) )     \\
    d-1 & (\rho > 1+  (\frac{d-8}{2} \vee 0)).
    \end{cases}
\end{align}

\begin{theorem}[Gaussian deconvolution]
\label{thm:so}
Let $d>2$ and let $F_z$ satisfy Assumption~\ref{ass:Fso}.
Then there exists a constant $c>0$ such that the
Fourier integral $\G_z$ of \eqref{eq:Gzint} satisfies
\begin{align} \label{eq:Gso}
\G_z(x) = \lambdaSO_z S_{\muSO_z}(x) + \begin{cases}
	O(\beta) 	& (x=0) \\
	O\( \frac{ \beta (L^{-c} + \beta) + \beta_1 }{ \abs x^{n_d} } \) & (x\ne 0),
\end{cases}
\end{align}
with the constants in the error term independent of $z, \beta_0, \beta_1, L$.
Moreover, for fixed $z, \beta_0, \beta_1, L$, the error term in \eqref{eq:Gso} can be replaced by $o(\abs x^{-n_d} )$ as $\abs x \to \infty$.
\end{theorem}

Theorem~\ref{thm:so} is related to \cite[Theorem~1.2]{HHS03}.
The critical case of Theorem~\ref{thm:so} can be inferred from the proof of \cite[(1.13)]{HHS03},
with $\beta$ taken to be of order $L^{-2+\eps}$,
$\eps \in (0, \frac{\rho\wedge 1}{2})$ (as in the proof
of \cite[Proposition~2.2]{HHS03}),
and with the error term in \eqref{eq:Gso} replaced by
\begin{equation}
    \frac{1}{L^{2-(\rho \wedge 2)}\xvee^{d-2+s}}
    =
    \frac{\beta L^{(\rho \wedge 2) -\eps }}{\xvee^{d-2+s}}
    \qquad (0 \le s  < \rho \wedge 2).
\end{equation}
For $s=n_d-(d-2)$ we have a smaller error term.
Our assumption that $\rho > \frac{d-8}{2}$ is not imposed in \cite{HHS03}, but this assumption is satisfied in all known applications.
Our proof of Theorem~\ref{thm:so} follows the method of \cite{LS24a}.
It is completely different from the analysis used
in \cite{HHS03}, and is simpler technically and conceptually.
By applying the fractional derivative method of \cite[Section~2.3]{LS24a}, the error term in \eqref{eq:Gso} can be replaced by $O(  \beta  \xvee^{-(d-2+s)} )$ for any $s < \rho \wedge 2 \wedge (\rho - \frac{ d-8 }2) $, with the constant independent of $z, \beta_0, \beta_1, L$.
We omit such improvement as we do not need it for the proof of Theorem~\ref{thm:Gcrit}.

\subsection{Proof of Theorem~\ref{thm:Gcrit}: bootstrap argument}
\label{sec:boot}

We now prove Theorem~\ref{thm:Gcrit} assuming Proposition~\ref{prop:so-nn},
Proposition~\ref{prop:general-1}, and Theorem~\ref{thm:so}.
The following proposition is the core of the bootstrap argument.

\begin{proposition}[Bootstrap]
\label{prop:boot}
Let $d > 2$.
Under Assumptions~\ref{ass:G} and \ref{ass:diagram},
there is an $L_2$ such that for all $L\ge L_2$, if
$z \in [1,z_c)$ and $b(z) \le 3$ then $b(z) \le 2$.
\end{proposition}

\begin{proof}
Suppose $z \in [1,z_c)$ satisfies $b(z)\le 3$.
By Assumptions~\ref{ass:diagram},
we have $ F_z * G_z = \delta$ with $F_z = \delta - zD - \Pi_z$.
We first verify Assumption~\ref{ass:Fso} for $F_z$:
\eqref{eq:Pi_decay} holds with
$\beta=\beta_0$ proportional to $L^{-2+\eps}$ and $\beta_1 = o(L^{-2+\eps})$ by \eqref{eq:Pibd}, and
\begin{align} \label{eq:Fz0}
\hat F_z(0) = \frac 1 {\hat G_z(0) } = \frac 1 { \sum_{x\in\Z^d} G_z(x) } \ge 0,
\end{align}
since $G_z$ is non-negative and summable when $z < z_c$.

We now use Theorem~\ref{thm:so}.
Since $G_z$ satisfies $F_z * G_z = \delta$ and both $F_z$ and $G_z$ are summable,
we have $\hat F_z(k)\hat G_z(k)=1$ and hence
$G_z$ is equal to the Fourier
integral $\G_z$ of \eqref{eq:Gzint}.
With $\lambdaSO_z = 1+O(\beta)$ from \eqref{eq:lambda1}, Theorem~\ref{thm:so} gives
\begin{align}
\label{eq:GCasy}
G_z(x) = (1+O(\beta) ) S_{\muSO_z}(x)
+ O\( \frac{ o(\beta) }{ \abs x^{d-2} } \)
\qquad (x\ne 0),
\end{align}
where the constant is independent of $z,\beta,L$.
By Proposition~\ref{prop:so-nn},
\begin{align}
\label{eq:GCbd}
S_{\muSO_z}(x) \le S_1(x)  \le \frac{ K_S }{ L^{2-\eps}\abs x^{d-2} }
	\qquad (x\ne 0).
\end{align}
Combined with $\beta = \const L^{-2 + \eps}$, we get
\begin{equation}
G_z(x)
\le \( 1+O(\beta) + o(1) \)
\frac{ K_S }{ L^{2-\eps}\abs x^{d-2} }
\le 2 \frac{ K_S }{ L^{2-\eps}\abs x^{d-2} }
	\qquad 	(x \neq 0),
\end{equation}
for sufficiently large $L$ (we emphasise that $L$ is taken large here
independently of $z$).
Also, since $z = 1 + O(\beta)$, we have $3(z-1) = O(\beta) \le 2$.
This proves that $b(z) \le 2$
for sufficiently large $L$.
\end{proof}

\begin{proof}[Proof of Theorem~\ref{thm:Gcrit}]
By Proposition~\ref{prop:general-1}, it suffices to consider the
impulse equation \eqref{eq:Glace},
so we work under Assumption~\ref{ass:diagram}.

By Proposition~\ref{prop:boot} and continuity of the function $b$, the interval $(2,3]$ is forbidden for values of $b(z)$ when $z\in[1, z_c)$.
Since $b(1) \le 1$ by Assumption~\ref{ass:G}(iv), we must have $b(z) \le 2$ for all $z \in [1, z_c)$.
It then follows from Assumption~\ref{ass:G}(iii) that
\begin{align}
\label{eq:Gcrit_bdd}
G\crit(x) = \lim_{z \to z_c^-} G_z(x) \le 2 \frac{ K_S }{ L^{2-\eps}\abs x^{d-2} }
	\qquad (x\ne 0).
\end{align}
The bound \eqref{eq:Gcrit_bdd} implies that $b(z_c) \le 2$, so Assumption~\ref{ass:diagram}
gives a critical $F\crit = \delta - z_c D - \Pi\crit$
with $\Pi\crit$ obeying \eqref{eq:Pibd}.
By monotone convergence, we can take
the $z\to z_c^-$ limit of \eqref{eq:Fz0} to see that $\hat F_{z_c}(0) = 0$, so now Assumption~\ref{ass:Fso} is verified at $z=z_c$.

For $z<z_c$,
it follows from $b(z) \le 2$ and $G_z$ being summable that $G_z = \G_z$,
as in the proof of Proposition~\ref{prop:boot}.
To prove that also $G\crit = \G\crit$,
by the $L^2$ Fourier transform it suffices to show that $G\crit \in \ell^2(\Z^d)$.
This follows from $d>4$, \eqref{eq:Gcrit_bdd}, and the fact that $G\crit(0) < \infty$.
The latter is a consequence of the $x=0$ case of Theorem~\ref{thm:so} applied to $G_z$ with $z<z_c$,
together with \eqref{eq:lambda1} and the bound on $S_1(0)$ from \eqref{eq:S1bd}:
\begin{align}
G\crit(0) = \lim_{z\to z_c^-} G_z(0)
\le (1 + O(\beta) ) S_1(0) + O(\beta) \le 1+O(\beta).
\end{align}

Since $\hat F_{z_c}(0) = 0$, we see from \eqref{eq:lambda1} that $\mu_{z_c}=1$.
Therefore, by Theorem~\ref{thm:so},  Proposition~\ref{prop:so-nn}, and \eqref{eq:C1_asymp}, we have
\begin{align}
G\crit(x) \sim \lambdaSO\crit S_1(x)
\sim \frac{ \lambdaSO\crit }{ \sigma^2 }C_1(x)
\sim \frac{ \lambdaSO\crit }{ \sigma^2 } \frac{ a_d }{ \abs x^{d-2} }
\qquad \text{as $\abs x \to \infty$},
\end{align}
which is the desired result.
By \eqref{eq:lambda1}, $\lambdaSO\crit = 1+O(\beta)$, and $\lambda_{z_c}$ is
given explicitly in terms of $\Pi_{z_c}$ as
\begin{equation}
\label{eq:lambda-critical}
    \lambdaSO_{z_c} =
    \frac{1}{z_c+\sigma^{-2}\sum_{x\in \Z^d} \abs x^2 \Pi_{z_c}(x)}.
\end{equation}

This completes the proof of Theorem~\ref{thm:Gcrit},
subject to Proposition~\ref{prop:so-nn},
Proposition~\ref{prop:general-1}, and Theorem~\ref{thm:so}.
\end{proof}

\begin{remark}
\label{rk:d2}
The assumption that $d>4$ in Theorem~\ref{thm:Gcrit} is used
only to justify that $G\crit(x)$ is equal to the Fourier integral $\G\crit(x)$
of \eqref{eq:Gzint} (as mentioned in the previous proof, we know that $G_z=\G_z$ for $z <z_c$).
If we assume, in addition to Assumption~\ref{ass:diagram}, that $\Pi_z(x)$ is left-continuous at $z=z_c$ for all $x$, then we can relax to all $d>2$, since then
the equality $G\crit(x) =\G\crit(x)$ follows from
the infrared bound \eqref{eq:FIR} together with the dominated convergence theorem to
take the limit $z \to z_c^-$ in \eqref{eq:Gzint}.
This additional continuity assumption can be verified in practice
(see \cite[Appendix~A]{Hara08}), but we do not
comment further since all our applications have $d>4$.
\end{remark}

\section{Proof of deconvolution Theorem~\ref{thm:so}}
\label{sec:proof-so}

We follow the strategy of \cite[Sections~2.2]{LS24a}, but
additional care is
required to track dependence on $\beta_0,\beta_1,L$.

\subsection{Fourier analysis}

For a function $g : \T^d \to \C$
and $p \in [1,\infty)$, we write
\begin{equation}
    \norm g _p^p = \int_{\T^d} |g(k)|^p \frac{dk}{(2\pi)^d}
\end{equation}
for the $L^p(\T^d)$ norm,
and we write $\|g\|_\infty$ for the supremum norm.

We first isolate the leading decay of $\G_z$.
Suppose $F_z$ obeys Assumption~\ref{ass:Fso}.
For $\mu\in(0,1]$, define $A_\mu = \delta - \mu D$,
so that $A_\mu * S_\mu = \delta$ by \eqref{eq:S}.
For any $\lambda\in \R$, we write
\begin{align} \label{eq:GC}
\G_z
&= \lambda S_\mu + \delta * \G_z - \lambda S_\mu * \delta	\nl
&= \lambda S_\mu + (S_\mu * A_\mu) * \G_z - \lambda S_\mu * (F_z * \G_z)	\nl
&= \lambda S_\mu + S_\mu * E_{z,\lambda,\mu} * \G_z,
\end{align}
with
\begin{equation}
\label{eq:Edef}
    E_{z, \lambda,\mu} = A_\mu - \lambda F_z.
\end{equation}
The choice of $\lambda_z, \mu_z$ in \eqref{eq:lambda_z-so}
has been made to ensure that
\begin{align} \label{eq:E_condition}
\sum_{x\in\Z^d} E_{z,\lambda_z,\muz}(x)
= \sum_{x\in\Z^d} \abs x^2 E_{z,\lambda_z,\muz}(x)
= 0.	
\end{align}
Indeed, \eqref{eq:E_condition} is a system of two linear equations in $\lambda, \mu$, with solution given by \eqref{eq:lambda_z-so}.
This isolates the leading term $\lambda_z S_\muz$:
\begin{align}  \label{eq:G_isolate}
\G_z &= \lambda_z S_\muz + f_z,
\qquad
f_z =  S_\muz * E_{z,\lambda_z,\muz} * \G_z.
\end{align}
We only use \eqref{eq:G_isolate} in its Fourier version, namely
\begin{align}
\label{eq:G_isolate_k}
\frac 1 { \hat F_z }
&= \frac {\lambda_z} { \hat A_\mu } + \hat f_z,
\qquad
\hat f_z  =  \frac{\hat E_{z,\lambda_z,\mu_z} }{ \hat A_{\mu_z} \hat F_z },
\qquad
\hat E_{z,\lambda_z,\mu_z} = \hat A_{\mu_z} - \lambda_z \hat F_z.
\end{align}

To simplify the notation, we will usually omit subscripts $z, \lambda_z, \muz$, and
use subscripts to denote partial derivatives instead, e.g., $\hat E_\alpha = \nabla^\alpha \hat E_{z, \lambda_z, \muz}$ for a multi-index $\alpha$.
The proof of Theorem~\ref{thm:so} replies on a classical fact about
the Fourier transform: smoothness of a function on $\T^d$ is related to the decay of its Fourier coefficient.  Concretely, we have the following lemma, which repeats \cite[Lemma~2.3]{LS24a} (an elementary proof is given in \cite{LS24a}).
Necessary properties of weak derivatives
are summarised in \cite[Appendix~A]{LS24a}.

\begin{lemma}
\label{lem:Graf}
Let $a,d>0$ be positive integers and let $h : \Z^d \to \R$.
There is a constant $c_{d,a}$, depending only on the dimension $d$ and the
maximal order $a$ of differentiation,
such that if the weak derivative $\hat h_\alpha$ is in $L^1(\T^d)$
for all multi-indices $\alpha$ with $|\alpha| \le a$ then
\begin{equation}
\label{eq:Graf}
    |h(x)|
    \le
    c_{d,a} \frac{1}{\xvee^a}
    \max_{|\alpha| \in \{0,a\}}\|\hat h_\alpha\|_1 .
\end{equation}
Moreover, $|x|^a h(x) \to 0$ as $|x|\to\infty$.
\end{lemma}

Recall from \eqref{eq:nddef} that,
assuming $\rho > \frac{d-8}{2} \vee 0$,
\begin{align} \label{eq:nddef-bis}
n_d = \begin{cases}
    d-2 & (\rho \le 1 + (\frac{d-8}{2} \vee 0) )     \\
    d-1 & (\rho > 1+  (\frac{d-8}{2} \vee 0)).
    \end{cases}
\end{align}
For later reference, we observe that $n_d$ is the largest integer that satisfies
\begin{equation} \label{eq:ndub}
n_d < (d - 2 + \rho \wedge 2 ) \wedge ( \tfrac 1 2 d + 2 + \rho ).
\end{equation}

\begin{proposition}
\label{prop:f_SO}
Let $F_z$ obey Assumption~\ref{ass:Fso}.
Then the function $\hat f_z$ defined in \eqref{eq:G_isolate_k}
is $n_d$ times weakly differentiable, and for any multi-index $\alpha$ with $\abs \alpha \le n_d$,
\begin{align}
\label{eq:fbeta}
\norm { \hat f_\alpha  }_{r} \lesssim \beta
	\qquad (r\inv > \frac{ \abs \alpha + 2 - \rho \wedge 2}d ),
\end{align}
with the constant independent of $z, \beta_0, \beta_1, L$.
Moreover, if $\abs \alpha \ne 0$, then
\begin{align} \label{eq:fbeta2}
\norm { \hat f_\alpha  }_{r} \lesssim \beta (L^{-c} + \beta) + \beta_1
	\qquad (r\inv > \frac{ \abs \alpha + 2 - \rho \wedge 2}d )
\end{align}
for some $c>0$, with the constants independent of $z,\beta_0, \beta_1, L$.
\end{proposition}

\begin{proof}[Proof of Theorem~\ref{thm:so} assuming Proposition~\ref{prop:f_SO}]
By Proposition~\ref{prop:f_SO}, $\hat f_z$ is $n_d$ times weakly differentiable, so $\grad^\alpha \hat f_z \in L^1(\T^d)$ for all multi-indices $\alpha$ with $\abs \alpha \le n_d$.
By \eqref{eq:ndub}, $r=1$ is permitted in \eqref{eq:fbeta} and \eqref{eq:fbeta2}
for all $|\alpha|\le n_d$.
Since $\G_z = \lambdaSO_z S_\muz + f_z$, it follows from Lemma~\ref{lem:Graf} and \eqref{eq:fbeta} that
\begin{align}
\G_z(x) = \lambdaSO_z S_\muz(x) + O \( \frac \beta {\xvee^{n_d}}\),
\end{align}
and that $f_z(x) = o(|x|^{-n_d})$ as $\abs x\to \infty$ for fixed $z, \beta_0, \beta_1, L$.
For the improved dependence of the error term on $\beta,L$ when $x\ne 0$, we note that the $\abs \alpha=0$ part of \eqref{eq:Graf} is only used to estimate $h(0)$.
Therefore, it suffices to observe that \eqref{eq:fbeta2} implies that
$\norm{ \hat f_\alpha }_1 \lesssim  \beta (L^{-c} + \beta) + \beta_1$
for some $c>0$ for all $\abs \alpha = n_d$.
\end{proof}

The proof of Proposition~\ref{prop:f_SO} uses product and quotient rules of differentiation.
Since $\hat f = \hat E / (\hat A \hat F)$, the $\alpha$-th weak derivative of $\hat f$ is given by a linear combination of terms of the form
\begin{align} \label{eq:f_decomp0}
\frac{ \prod_{n=1}^i \hat A_{\delta_n} }{\hat A ^{1+i} }
\hat E_{\alpha_2}
\frac{ \prod_{m=1}^j \hat F_{\gamma_m}  }{\hat F^{1+j} }
=
\( \prod_{n=1}^i \frac{ \hat A_{\delta_n} }{\hat A } \)
\( \frac{ \hat E_{\alpha_2} }{ \hat A \hat F } \)
\( \prod_{m=1}^j \frac{ \hat F_{\gamma_m}  }{ \hat F } \),
\end{align}
where $\alpha = \alpha_1 + \alpha_2 + \alpha_3$, $ 0 \le i \le \abs {\alpha_1}$, $0 \le j \le \abs {\alpha_3 }$, $\sum_{n=1}^i \delta_n = \alpha_1$, and $\sum_{m=1}^j \gamma_m = \alpha_3$,
provided that
we can justify \emph{a posteriori} that all terms of the form \eqref{eq:f_decomp0} are integrable (see \cite[Appendix~A]{LS24a}).
For this, we use H\"older's inequality and the following lemma.

\begin{lemma}
\label{lem:AFEbds_SO}
Let $F_z$ obey Assumption~\ref{ass:Fso}.
Let $\gamma$ be a multi-index with $ \abs \gamma < (d-2+\rho \wedge 2) \wedge (\tfrac 1 2 d  + 2 +\rho)$.
Choose $\sigma \in (0, \rho \wedge 2)$
and $q_1,q_2$ satisfying
\begin{equation} \label{eq:q1q2}
    \frac{ \abs \gamma } d < q_1\inv < 1,
    \qquad
    \frac{ 2 - \sigma + \abs \gamma } d < q_2\inv < 1.
\end{equation}
Then $\hat F, \hat A, \hat E$ are $\gamma$-times weakly differentiable, and
\begin{align} \label{eq:AFEbds_SO}
\Bignorm{ \frac{ \hat A_\gamma }{ \hat A } }_{q_1} , \;
\Bignorm{ \frac{ \hat F_\gamma }{ \hat F } }_{q_1}
	\lesssim 1 , \qquad
\Bignorm{ \frac{ \hat E_\gamma }{ \hat A \hat F }}_{q_2}
\lesssim \beta,
\end{align}
with the constants independent of $z,\beta_0, \beta_1, L$.
Moreover, if $\abs \gamma \ne 0$, the estimates are improved to
\begin{align} \label{eq:AFEbds_SO2}
\Bignorm{ \frac{ \hat A_\gamma }{ \hat A } }_{q_1}
	\lesssim L^{ \abs \gamma - d/q_1 } 	, \quad
\Bignorm{ \frac{ \hat F_\gamma }{ \hat F } }_{q_1}
	\lesssim L^{ \abs \gamma - d/q_1 } + \beta , \quad
\Bignorm{ \frac{ \hat E_\gamma }{ \hat A \hat F }}_{q_2}
	\lesssim \beta L^{2 - \sigma + \abs \gamma - d/q_2} + \beta_1,
\end{align}
with the constants independent of $z,\beta_0, \beta_1, L$.
\end{lemma}

\begin{proof}[Proof of Proposition~\ref{prop:f_SO} assuming Lemma~\ref{lem:AFEbds_SO}]
Let $\abs \alpha \le n_d$, $\rho_2 = \rho \wedge 2$, and pick some $\sigma \in (0,\rho_2)$.
We use the product and quotient rules of weak derivatives \cite[Lemmas~A.2--A.3]{LS24a} to calculate $\hat f_\alpha$.
For the hypotheses of these rules, we need to verify all terms of the form \eqref{eq:f_decomp0} are integrable.
By H\"older's inequality and Lemma~\ref{lem:AFEbds_SO}, \eqref{eq:f_decomp0} belongs to $L^r(\T^d)$ as long as
\begin{align}
\frac 1 r > \frac{ \sum_{n=1}^i \abs {\delta_n} } d + \frac{ 2 - \sigma + \abs{\alpha_2 } }d
	+ \frac{ \sum_{m=1}^j \abs {\gamma_m} } d
= \frac{ \abs \alpha + 2 - \sigma } d.
\end{align}
Since $\sigma < \rho_2$ is arbitrary, this shows that $\eqref{eq:f_decomp0}$ is in $L^r$ for all
$r\inv > (\abs \alpha + 2 - \rho_2 )/d$.
In particular, it belongs to $L^1$ since $\abs \alpha \le n_d < d-2+\rho_2$ by \eqref{eq:ndub}.
This proves that $\hat f$ is $\alpha$-times weakly differentiable and that $\hat f_\alpha \in L^r$ with the same values of $r$.
Furthermore, we get a quantitative estimate on $\norm {\hat f_\alpha}_r$ from H\"older's inequality.
For \eqref{eq:fbeta}, we use \eqref{eq:AFEbds_SO} and get $\beta$ from the norm of $\hat E_{\alpha_2} / (\hat A \hat F)$.
For \eqref{eq:fbeta2}, since there is at least one derivative taken,
in one of the factors we can use the stronger \eqref{eq:AFEbds_SO2}. The constant $c>0$ is produced by the strict inequalities in \eqref{eq:q1q2}.
\end{proof}

To complete the proof of Theorem~\ref{thm:so},
it remains to prove Lemma~\ref{lem:AFEbds_SO}.
The proof uses the following elementary facts about the Fourier transform.
The first lemma translates the good moment behaviour of $E(x)$ in \eqref{eq:E_condition} (the first moments of $E(x)$ also vanish, by symmetry) into good bounds on $\hat E(k)$ and its derivatives, which ultimately allows us to take $n_d$ derivatives of $\hat f$.
The second lemma uses boundedness of the $L^p$ Fourier transform when $1\le p \le 2$.

\begin{lemma}[{\cite[Lemma~2.2]{LS24a}}]
\label{lemma:Ek}
Suppose $E: \Z^d \to \R$ is $\Z^d$-symmetric,
has vanishing zeroth and second moments as in
\eqref{eq:E_condition},
and satisfies $\abs{ E(x) } \le  K \abs x^{-(d+2 + \rho)}$ for some $K, \rho>0$.
Choose $\sigma \in (0, \rho)$ such that $\sigma \le 2$ and let $\alpha$ be a multi-index
with
$\abs \alpha < 2 + \sigma$.
Then there is a constant $c = c(\sigma, \rho, d)$ such that
\begin{align} \label{eq:Ek}
\abs{ \hat E_\alpha (k) }
\le cK \cdot \abs k^{2+\sigma - \abs \alpha}.
\end{align}
\end{lemma}

\begin{lemma}[{\cite[Lemma~2.6]{LS24a}}]
\label{lem:FTbdd}
Let $h: \Z^d \to \R$ obey $|h(x)| \le K \nnnorm x^{-b}$ for some $K,b>0$.
\begin{enumerate}[label=(\roman*)]
\item
If $b>d$ then $h \in \ell^1(\Z^d)$, $\hat h \in L^\infty(\T^d)$, and $\norm{ \hat h }_\infty \le c_{d,b} K$.

\item
If $b\le d$ then $h \in \ell^p(\Z^d)$ for $p>d/b$.
If also $\frac d2 < b \le d$ then $\hat h \in L^q(\T^d)$
and $\norm{ \hat h } _q \le c_{d,b,q}K$  for all $1 \le q < d/(d-b)$.
\end{enumerate}
\end{lemma}

\subsection{Proof of Lemma~\ref{lem:AFEbds_SO}}

We first collect properties of $D$ that we need.  The proof of Lemma~\ref{lemma:D_L}
is deferred to Appendix~\ref{app:D}.

\begin{lemma} \label{lemma:D_L}
If $L\ge L_0$ with $L_0$ sufficiently large (depending only on $d,v$),
then the following statements hold.
For any $a>0$,
\begin{align}
\label{eq:D_decay}
D(x) \lesssim \frac{  L^{a} } { \xvee^{d + a} }.
\end{align}
Uniformly in $\muSO \in [\frac 12,1]$,
$\hat A_\mu = 1 - \muSO\hat D$ satisfies the \emph{infrared bound}
\begin{align}
\label{eq:AIR}
\hat A_\muSO(k) - \hat A_\muSO(0) \gtrsim  L^2 \abs k^2 \wedge 1
\qquad (k \in \T^d).
\end{align}
For each multi-index $\alpha$,
\begin{align}
\label{eq:D_alpha}
\norm {  \hat D_\alpha }_{q } \lesssim L^{\abs\alpha - d/q} 	
\qquad ( 0 \le q\inv < 1 ).
\end{align}
\end{lemma}

Together with \eqref{eq:AIR},
Assumption~\ref{ass:Fso} implies an infrared bound for $\hat F_z$.  To see this, we first write
\begin{equation}
\hat F_z(k) - \hat F_z(0)
= z ( 1 - \hat D(k) ) + ( \hat \Pi_z(0) - \hat \Pi_z(k) ).
\end{equation}
The first term is bounded from below using $z \ge 1$ and \eqref{eq:AIR},
and the second term is bounded in absolute value by $O(\beta) ( \abs k^2 \wedge 1 )$, by Taylor's theorem, symmetry, and \eqref{eq:Pi_decay}.
Since $\beta$ is small, this yields
\begin{align}
\label{eq:IR-so}
\hat F_z(k) - \hat F_z(0) \ge K_{\rm IR} (L^2 \abs k^2 \wedge 1)  \qquad (k \in \T^d)
\end{align}
for some $K_{\rm IR} > 0$.
Because of the two alternatives in the infrared bounds \eqref{eq:AIR} and \eqref{eq:IR-so},
we pay separate attention to the small ball
\begin{align}
\label{eq:BLdef}
    B_L & = \{k \in \R^d : \|k\|_\infty \le \pi,\; |k|<1/L\},
\end{align}
and to its complement.

\begin{proof}[Proof of Lemma~\ref{lem:AFEbds_SO}]
\emph{Bound on $\hat A_\gamma/\hat A$.}
There is nothing to prove for $\abs \gamma = 0$ since the ratio is then 1.
We will prove that, for $\abs \gamma \ge 1$,
\begin{align}
\label{eq:AA_L}
\Bignorm{ \frac{ \hat A_\gamma }{ \hat A } }_{q}
\lesssim L^{\abs \gamma - d/q }
\qquad (  \frac{ \abs \gamma \wedge 2 } d < q\inv < 1 ).
\end{align}
This is stronger than the desired \eqref{eq:AFEbds_SO2} by allowing more values of $q = q_1$.
It also implies $\norm{ \hat A_\gamma / \hat A }_{q} \lesssim 1$ when we restrict to $ \abs \gamma / d < q\inv < 1$.
By the infrared bound \eqref{eq:AIR},
\begin{equation}
\label{eq:AAIR}
    \Big| \frac{ \hat A_\gamma }{ \hat A }(k) \Big|
    \lesssim
    L^{-2}|k|^{-2} |\hat A_\gamma(k)| \1_{B_L} + |\hat A_\gamma(k)|,
\end{equation}
where $B_L$ is the small ball in \eqref{eq:BLdef}.
Since $\hat A_\gamma = - \muSO \hat D_\gamma$, by \eqref{eq:D_alpha}
the $L^q$ norm of the second term on the right-hand side is bounded by $L^{\abs \gamma - d/q}$, as required.
If $\abs \gamma = 1$, Taylor's Theorem and symmetry give $\abs { \hat A_\gamma (k) } \lesssim L^2  \abs k$, so the first term is bounded by $|k|^{-1}\1_{B_L}$, which has $L^q$ norm bounded by $L^{1 - d/q}$ for $q\inv > 1/ d$.
For the remaining case $\abs \gamma \ge 2$, it follows from
H\"older's inequality and \eqref{eq:D_alpha} that
\begin{align}
\Bignorm{ \frac{ \hat A_\gamma }{ \hat A } \1_{B_L}}_{q}
\lesssim \frac 1 {L^2} \norm{ \hat A_\gamma }_{r} \bignorm{  \abs k ^{-2} \1_{B_L} }_{p}
&\lesssim \frac 1 {L^2} (L^{\abs \gamma - d/r} ) L^{2 - d/p}
= L^{\abs \gamma - d/q}
\end{align}
for $q\inv = r \inv + p \inv$ with $ 0\le r\inv < 1$ and $p \inv > 2/d$,
which in particular holds for any $q\inv > 2/d$.
This completes the proof of \eqref{eq:AA_L}.

\medskip\noindent \emph{Bound on $\hat F_\gamma/\hat F$.}
The $\abs \gamma = 0$ case is again trivial.
We will prove that if $1 \le \abs \gamma < \half d + 2 + \rho$ then
\begin{align}
\label{eq:FF_L}
\Bignorm{ \frac{ \hat F_\gamma }{ \hat F } }_{q}
\lesssim L^{\abs \gamma - d/q } + \beta 	\lesssim 1
\qquad (  \frac{ \abs \gamma } d < q\inv < 1 ).
\end{align}
The second inequality holds because $\beta$ is small.
As in \eqref{eq:AAIR}, by the infrared bound \eqref{eq:IR-so},
\begin{equation} \label{eq:FF_decomp}
    \Big| \frac{ \hat F_\gamma }{ \hat F }(k) \Big|
    \lesssim
    L^{-2}|k|^{-2} |\hat F_\gamma(k)| \1_{B_L} + |\hat F_\gamma(k)|.
\end{equation}
Since $\hat F = 1 - z \hat D - \hat \Pi$ and $\abs \gamma \ge 1$, by the triangle inequality,
by the fact that $z \le O(1)$ by \eqref{eq:zbd},
by \eqref{eq:D_alpha}, and by the Fourier transform bound
Lemma~\ref{lem:FTbdd}(ii) applied with
$h(x)=i^{|\gamma|}x^\gamma \Pi(x)$,
\begin{align}
\label{eq:F_gamma_L}
\norm{ \hat F_\gamma  }_q
\le \abs z \norm { \hat D_\gamma }_q  + \norm{ \hat \Pi _\gamma }_q
\lesssim  L^{\abs \gamma - d/q} + \beta
	\qquad ( 0 \le q\inv < 1, \  q \inv > \frac{ \abs \gamma - 2 - \rho }d ).
\end{align}
This gives the desired bound for the second term of \eqref{eq:FF_decomp}.
If $\abs \gamma = 1$, Taylor's Theorem and symmetry give $\abs { \hat F_\gamma (k) } \lesssim (L^2+\beta)  \abs k$, so the first term is bounded by $(1 + \beta/L^2) \abs k \inv \1_{B_L}$, which has $L^q$ norm bounded by $(1+\beta/L^2)L^{1 - d/q} \lesssim L^{1 - d/q}$ when $q\inv > 1/ d$.
For the remaining case $\abs \gamma \ge 2$, we let $r\inv = (\abs \gamma - 2)/d$. It follows from H\"older's inequality and \eqref{eq:F_gamma_L} that
\begin{align}
\Bignorm{ \frac{ \hat F_\gamma }{ \hat F } \1_{B_L}}_q
&\lesssim \frac 1 {L^2} \norm{ \hat F_\gamma }_r \bignorm{  \abs k ^{-2} \1_{B_L} }_p
\lesssim \frac 1 {L^2} ( L^{\abs \gamma - d/r} +\beta ) L^{2 - d/p}
\lesssim L^{\abs \gamma - d/q}
\end{align}
for $q\inv = r \inv + p \inv$ and $p \inv > 2/d$. By the choice of $r$, the bound holds for $q\inv > \abs \gamma / d$. This completes the proof of \eqref{eq:FF_L}.

\medskip\noindent \emph{Bound on $\hat E_\gamma / \hat A \hat F$.}
Let $\abs \gamma < \half d + 2 + \rho$ and choose $0 < \sigma < \rho \wedge 2 $.
Our goal is to prove that
\begin{align}
\label{eq:EAF_L}
\Bignorm{ \frac{ \hat E_\gamma }{ \hat A \hat F }}_q
\lesssim \beta
\qquad ( \frac{ 2 - \sigma + \abs \gamma } d < q\inv < 1 ),
\end{align}
which establishes \eqref{eq:AFEbds_SO},
and to improve the bound when $\abs \gamma \ne 0$ to
\begin{align} \label{eq:EAF_boot}
\Bignorm{ \frac{ \hat E_\gamma }{ \hat A \hat F }}_q
\lesssim \beta   L^{2-\sigma + \abs\gamma - d/q} +  \beta_1
\qquad ( \frac{ 2 - \sigma + \abs \gamma } d < q\inv < 1 ),
\end{align}
which is \eqref{eq:AFEbds_SO2}.

It follows from the formula $\hat E = \hat A_{\mu} - \lambda \hat F$ in \eqref{eq:G_isolate_k},
together with the fact that $\mu = 1-\lambda \hat F(0)$ by \eqref{eq:lambda_z-so}, that
\begin{equation}
\label{eq:EPi}
    \hat E = (1-\lambdaSO)(1- \hat D) - \lambdaSO \hat \Pi(0) \hat  D + \lambdaSO \hat \Pi.
\end{equation}
Also, by the infrared bounds for $\hat A$ and $\hat F$,
\begin{equation}
\label{eq:EAF_decomp}
    \Big| \frac{ \hat E_\gamma }{ \hat A \hat F }(k) \Big|
    \lesssim
    L^{-4}|k|^{-4} |\hat E_\gamma(k)| \1_{B_L} + |\hat E_\gamma(k)|.
\end{equation}
We will first show that the second term on the right-hand side of
\eqref{eq:EAF_decomp} obeys \eqref{eq:EAF_L} and \eqref{eq:EAF_boot}, and then show that
the first term on the right-hand side of
\eqref{eq:EAF_decomp} obeys the stronger bound \eqref{eq:EAF_boot} even when
$|\gamma|=0$.

For the second term on the right-hand side of \eqref{eq:EAF_decomp},
we use \eqref{eq:EPi}, the triangle inequality,
$\lambdaSO = 1 + O(\beta)$ (by \eqref{eq:lambda1}), and $|\hat \Pi(0)|\lesssim\beta$,
to see that
\begin{align} \label{eq:E_gamma_L0}
\norm{ \hat E_\gamma  }_q
\lesssim
O(\beta) ( 1 + \norm { \hat D_\gamma }_q ) + \abs \lambdaSO \norm{ \hat \Pi _\gamma }_q .
\end{align}
By \eqref{eq:D_alpha}, $\norm { \hat D_\gamma }_{q} \lesssim L^{\abs \gamma - d/q}$.
We bound the norm of $\hat \Pi_\gamma$
using \eqref{eq:Pi_decay} and Lemma~\ref{lem:FTbdd}, to see that
\begin{align}
\label{eq:E_gamma_L}
\norm{ \hat E_\gamma  }_{q}
\lesssim \beta  ( 1 + L^{\abs \gamma - d/q} )
	\qquad ( 0 \le q\inv < 1,\ q\inv > \frac{ \abs \gamma - 2 - \rho} d).
\end{align}
In particular, \eqref{eq:E_gamma_L} holds for $( 2 - \sigma + \abs \gamma) / d < q\inv < 1$,
and the norm is bounded by a multiple of $\beta$ since $\sigma \le 2$.
To improve the bound when $\abs \gamma \ne 0$, we still use  \eqref{eq:EPi},
but now the contribution from the term $(1-\lambda)(1)$
vanishes because there is at least one derivative taken.
For the same reason, we can write $\hat \Pi _\gamma = \grad^\gamma ( \hat \Pi - \Pi(0) )$.
Since $\hat \Pi - \Pi(0)$ is the Fourier transform of $\Pi(x) - \Pi(0) \delta_{0,x}$, Lemma~\ref{lem:FTbdd} and the $x\ne0$ part of \eqref{eq:Pi_decay} imply that
\begin{align}
\norm{ \hat E_\gamma } _q
\lesssim \beta \norm{ \hat D_\gamma }_q +  \norm{ \hat \Pi_\gamma }_q
\lesssim \beta L^{\abs \gamma - d/q} + \beta_1,
	\qquad ( 0 \le q\inv < 1,\ q\inv > \frac{ \abs \gamma - 2 - \rho} d),
\end{align}
which implies that $|\hat E_\gamma|$ obeys
the upper bound in \eqref{eq:EAF_boot}, since $\sigma < \rho$ and $\sigma \le 2$.

For the first term on the right-hand side of \eqref{eq:EAF_decomp}, we will prove that
the stronger bound \eqref{eq:EAF_boot} holds for all $\gamma$.
Consider first the case $\abs \gamma < 2 + \sigma$.
It follows from the $x$-space version of \eqref{eq:EPi}, together with
the decay of $D$ in \eqref{eq:D_decay} and of $\Pi$ in \eqref{eq:Pi_decay}
that
\begin{align}
\label{eq:E_decay}
\abs{ E(x) } \lesssim \frac{ \beta  L^{2 + \rho_2} }{ \xvee^{d+2+\rho_2} }
\end{align}
(we have relaxed the decay of $\Pi$ because it is costly to make $D$ decay).  Then Lemma~\ref{lemma:Ek} with $\rho_2$ in place of $\rho$ gives
(it is here that we require the strict inequality $\sigma < \rho_2$)
\begin{align}
    | \hat E_\gamma(k) |
    \lesssim  \beta  L^{2+\rho_2} \abs k^{2+\sigma - \abs \gamma} .
\end{align}
The $L^q$ norm of $\abs k^{2+\sigma-\abs\gamma -4} = \abs k^{-(2-\sigma + \abs \gamma)}$ on $B_L$ is of order $L^{2 - \sigma + \abs \gamma - d/q}$, so the $L^q$ norm of the first term on
on the right-hand side of \eqref{eq:EAF_decomp} is bounded, in this case, by
the desired
\begin{align}
\frac{ \beta L^{2+\rho_2} }{ L^4 } L^{2 - \sigma + \abs \gamma - d/q}
&\le  \beta L^{2 - \sigma + \abs \gamma - d/q} \le \beta
	\qquad (q\inv > \frac{ 2 - \sigma + \abs \gamma } d).
\end{align}
For the remaining case $2 + \sigma \le \abs \gamma < \half d + 2 + \rho$, H\"older's inequality and \eqref{eq:E_gamma_L} imply that the $L^q$ norm of the first term on the
right-hand side of \eqref{eq:EAF_decomp} is bounded above by
\begin{align}
\frac 1 {L^4} \norm{ \hat E_\gamma }_r \bignorm{  \abs k ^{-4} \1_{B_L} }_p
&\lesssim \frac 1 {L^4} [ \beta (1+L^{\abs \gamma - d/r}) ] L^{4 - d/p}
\lesssim \beta L^{\abs \gamma - d/q}
\end{align}
for $q\inv = r \inv + p \inv$, $r\in (1, \infty]$,
$ \frac { \abs \gamma - 2 - \rho } d < r \inv \le  \frac { \abs \gamma }  d$, and $p \inv > 4/d$. In particular, since $\sigma < \rho$, the bound holds for all $q\inv > ( 2 - \sigma + \abs \gamma)/d$. The desired bound \eqref{eq:EAF_boot} for
the first term on the right-hand side of \eqref{eq:EAF_decomp}
then follows from $\sigma \le 2$.
This completes the proof of \eqref{eq:EAF_L}
and concludes the proof of the lemma.
\end{proof}

\section{Inhomogeneous deconvolution: proof of Proposition~\ref{prop:general-1}}
\label{sec:general}

We now prove Proposition~\ref{prop:general-1}, which concerns the
inhomogeneous convolution equation
\begin{equation}
\label{eq:Hh4}
    \HH_z = h_z + zD*h_z*\HH_z.
\end{equation}
In fact we prove a stronger proposition, with
arbitrary small
$\beta_0$ and $\beta_1$ rather than the specific choices in \eqref{eq:h_decay-1}.
We write $\theta=2+\rho$.

\begin{proposition}
\label{prop:general}
Suppose the function $h_z: \Z^d \to \R$ is
$\Z^d$-symmetric and satisfies
\begin{align} \label{eq:h_decay}
\abs{ h_z(x) - \delta_{0,x} }
\le \beta_0 \delta_{0,x} + \frac{ \beta_1 } { \xvee^{d+\theta}  }
\end{align}
with $\theta >  0$  and with
$\beta=\beta_0 \vee \beta_1 \ge 0$ sufficiently small.
Then there exists a $\Z^d$-symmetric function $\Phi_z : \Z^d \to \R$
for which $\HH_z$ of \eqref{eq:Hh4} satisfies $F_z * \HH_z = \delta$ with
\begin{align} \label{eq:h}
F_z = \delta - zD - \Phi_z,
\qquad  \abs{ \Phi_z(x) } \le  O(\beta) \delta_{0,x} + \frac{ O(\beta_1) } { \xvee^{d+\theta}  }.
\end{align}
\end{proposition}

The proof uses a Banach algebra, as in \cite{BHK18}.
Given $\zeta>0$, we define a norm on functions $v: \Z^d \to \R$ by
\begin{align}
\label{eq:def:norm}
\norm v_\zeta
= \max \Big \{ 2^{\zeta + 1 }  \sum_{x\in\Z^d} \abs {v(x)}, \;  \sup_{x\in \Z^d}  \abs x^\zeta  \abs{ v(x) } \Big \}.
\end{align}
Then $ \norm{ u*v}_\zeta \le \norm u _ \zeta \norm v_\zeta$ for all $u$ and $v$, so the space $\{ v : \norm v _ \zeta < \infty \}$ is a Banach algebra with
product given by convolution.
If $\norm{ v - \delta }_\zeta < 1 $, then $v$ has a deconvolution
$v\inv =\sum_{n=0}^\infty (\delta -v)^{*n}$
given by a convergent Neumann series.  Indeed, by writing $v=\delta-(\delta -v)$, we have
\begin{align}
    v*v^{-1}  & = \sum_{n=0}^\infty (\delta -v)^{*n} - \sum_{n=1}^\infty (\delta -v)^{*n}
    =\delta.
\end{align}
Also, $v\inv$ satisfies
\begin{align}
\label{eq:Neumann}
\norm {v\inv - \delta }_\zeta
& \le
\sum_{n=1}^\infty \|\delta -v\|_\zeta^{n} =
\frac{\norm{ v - \delta }_\zeta}{1 - \norm{ v - \delta }_\zeta}.
\end{align}

\begin{proof}[Proof of Proposition~\ref{prop:general}]
We drop subscripts $z$ from the notation as they play no role in the proof.
By \eqref{eq:h_decay},
we have $\norm{ h - \delta }_\zeta \le O(\beta) < 1$ with $\zeta = d+ \theta$, so the deconvolution $h^{-1}$ exists and \eqref{eq:Neumann} holds with $v = h$.
Define $F = \delta - zD - \Phi$ with $\Phi = \delta - h\inv$.
Then
\begin{align} \label{eq:Fz_equiv}
F = (\delta  - \Phi) - zD
= h\inv - zD
= h\inv * (\delta - zD*h),
\end{align}
so that, using \eqref{eq:Hh4} in the second equality,
\begin{align}
F * \HH
= h\inv * (\HH - zD*h * \HH)
= h\inv * h = \delta.
\end{align}

For the decay of $\Phi$, we first use $\norm{ h-\delta }_\zeta \le O(\beta)$, \eqref{eq:Neumann}, and the definition of $\norm{\cdot}_\zeta$ to get
\begin{equation} \label{eq:hinvbd}
\max \Big \{ 2^{\zeta+1} \abs{ \Phi (0) },\; \sup_{x\in \Z^d} \abs x^\zeta \abs{ \Phi(x) } \Big \}
\le \norm{ \Phi }_\zeta
= \norm{ h \inv - \delta }_\zeta
\le \frac{ O(\beta) }{ 1-O(\beta) }.
\end{equation}
This proves that
$\abs{ \Phi (x)  } \le  O(\beta) \nnnorm x^{-\zeta} $ and gives the
bound on $\Phi(0)$ in \eqref{eq:h}, but for $x \neq 0$ we wish to improve the $O(\beta)$
to $O(\beta_1)$.

We make the improvement as follows.
Let $f=\delta -h$, so that
\begin{align}
\Phi = \delta - h\inv = - \sum_{n=1}^\infty (\delta - h)^{*n}
	= - \sum_{n=1}^\infty f^{*n}.
\end{align}
For $x \neq 0$, we have $\Phi(x) = \Phi(x) - \Phi(0)\delta(x)$.
With $g_n=f^{*n}-f^{*n}(0)\delta$,
we write this as
\begin{equation}
\label{eq:Phi0}
    \Phi(x) - \Phi(0)\delta(x) = -\sum_{n=1}^\infty g_n(x)
    \qquad (x \neq 0).
\end{equation}
By definition, for $n \geq 1$,
\begin{equation}
    g_{n+1}(x)=(g_n*f)(x) + f^{*n}(0)g_1(x)
    \qquad (x\ne 0).
\end{equation}
Since $g_{n+1}(0)=0$ and $|f^{*n}(0)| \le \|f^{*n}\|_\zeta \le \|f\|_\zeta^n$,
it follows from the triangle inequality that
\begin{equation}
    \|g_{n+1}\|_\zeta
    \le
    \|f\|_\zeta\|g_n\|_\zeta   + \|f\|_\zeta^{n} \|g_1\|_\zeta.
\end{equation}
It then follows by iteration that
\begin{equation}
    \|g_{n}\|_\zeta
    \le
    n\|f\|_\zeta^{n-1}  \|g_1\|_\zeta
    \qquad (n \ge 1).
\end{equation}
Since we have $\norm f_\zeta \le O(\beta)$ and $\norm{ g_1 }_\zeta \le O(\beta_1)$ by the definition of $f$ and \eqref{eq:h_decay},
we can bound the norm of the sum in \eqref{eq:Phi0} as
\begin{equation}
    \|\Phi - \Phi(0)\delta \|_\zeta \le O(\beta_1),
\end{equation}
which implies the desired bound
 $\abs{ \Phi(x) } \le O(\beta_1) \abs x^{-\zeta}$ for $x\ne 0$.
This concludes the proof.
\end{proof}

For lattice trees and lattice animals, we cannot verify \eqref{eq:h_decay} directly, because
their one-point function plays a role without counterpart in the other models.
However, only a small adjustment is needed to apply our results.
The lace expansion for the two-point function $T_p$ of these models has the form
\begin{equation}
    T_p = t_p + pD * t_p * T_p.
\end{equation}
We divide out the one-point function as in \cite[(1.23)]{HHS03}.
We set $\tau_p = t_p(0)$, $z=p\tau_p$, $h_z=t_p/\tau_p$, and $\HH_z= T_p/\tau_p$.
This transforms the above equation to \eqref{eq:Hh4}
with $h_z$ now a small perturbation of $\delta$,
and Proposition~\ref{prop:general} can be applied.

\begin{appendix}

\section{Spread-out Green function}
\label{sec:so-Green}

We now  prove Proposition~\ref{prop:so-nn}, which asserts that
if $D$ is given by Definition~\ref{def:D}, 
if $d>2$, if $L \ge L_0$, and if $\eps>0$, then
the critical spread-out Green function $S_1(x)$ satisfies
\begin{align}
\label{eq:CL-uni3}
    S_1(x) & =  \delta_{0,x} + \frac{1}{\sigma^2} C_1(x)
    +O\( \frac{1}{L^{1-\eps} \xvee^{d-1}} \)
\end{align}
and $S_1(x) \leq  \delta_{0,x}+ K_S L^{-(2-\eps)} \xvee^{-(d-2)}$ for some $K_S=K_S(\eps)$,
with constants uniform in $L$ but dependent on $\eps$.
Our proof uses the same steps used to prove Theorem~\ref{thm:so}, and is conceptually and technically simpler than the proof
using intricate Fourier analysis in \cite{HHS03}.
We use the conclusion of Lemma~\ref{lemma:D_L} repeatedly. 
Although Lemma~\ref{lemma:D_L} is proved later, there
is no circularity because the proof of Lemma~\ref{lemma:D_L} is independent of
the proofs here.

We isolate the leading term as follows.
Recall that $\Dnn(x) = \frac{1}{2d} \1_{|x|=1}$ and that $\sigma^2$ is the variance of
$D$.
Let $A=\delta-\Dnn$, $F=\delta-D$,
$E=A - \sigma^{-2}
F$,
and $f=C_1*E*S_1$.
Then a similar calculation as in \eqref{eq:GC},
with $\lambda = \sigma^{-2}$, with $S_\mu$ replaced by $C_1$, and
with $G_z$ replaced by $S_1$, gives
\begin{equation}
S_1 = \sigma^{-2} C_1 + f,
\end{equation}
with $E(x)$ having vanishing zeroth and second moments as in \eqref{eq:E_condition}.
We further extract a Kronecker delta using $(\delta - D) * S_1 = \delta$, to get
\begin{align}
\label{eq:G_decomp_S-intro}
S_1 &= \delta + D *S_1
= \delta  + D*(\sigma^{-2} C_1 + f)  	
= \delta +\sigma^{-2} C_1 + \varphi ,
\end{align}
where
\begin{align}
\label{eq:phidef-intro}
    \varphi &=  D*f  - \sigma^{-2} C_1*F .
\end{align}
Now the error term $\varphi$ plays the role played by $f$ in Section~\ref{sec:proof-so}.

\begin{lemma}\label{lemma:phi}
Let $\eps>0$.
If $L\ge L_0$ with $L_0$ sufficiently large (depending only on $d,v$),
then the following statements hold.
The function $\hat \varphi$ is $d-1$ times weakly differentiable, and for any multi-index $\alpha$ with $\abs \alpha \le d-1$,
\begin{align}
\label{eq:phibd}
\norm{ \hat \varphi_\alpha }_1 \lesssim L^{-(1-\eps)},
\end{align}
with the constant independent of $L$ but depends on $\eps$.
\end{lemma}

\begin{proof}[Proof of Proposition~\ref{prop:so-nn} assuming Lemma~\ref{lemma:phi}]
It follows from Lemma~\ref{lem:Graf} and \eqref{eq:phibd}
that $\varphi(x) = O(L^{-(1-\eps)} \nnnorm x^{-(d-1)} )$, which implies \eqref{eq:CL-uni3}.
Also, the bound  $S_1(x) \leq  \delta_{0,x}+ K_S L^{-(2-\eps)} \xvee^{-(d-2)}$
follows from \eqref{eq:CL-uni3} and the asymptotic formula for $C_1(x)$ in \eqref{eq:C1_asymp} when $\abs x > L$, and follows from $S_1(x) = \delta_{0,x} + O(L^{-d})$
 when $\abs x \le L$ (this simple fact is proved in \cite[Section~6.1]{HHS03}).
\end{proof}

To prove Lemma~\ref{lemma:phi}, we use the following analogue of
Lemma~\ref{lem:AFEbds_SO}.
Its proof involves only a minor
adaptation of the proof of Lemma~\ref{lem:AFEbds_SO}.
(We write ${\tau}$ in place of the $\sigma$ in Lemma~\ref{lem:AFEbds_SO} because here
we reserve $\sigma^2$ for the variance of $D$.)

\begin{lemma}
\label{lem:AFEbds_S}
Let $L\ge L_0$ with $L_0$ sufficiently large (depending only on $d,v$).
Let $\gamma$ be a multi-index with $\abs \gamma < d$. 
Choose $\tau \in (0,2)$ and $q_1,q_2$ satisfying
\begin{equation}
    \frac{ \abs \gamma } d < q_1\inv < 1,
    \qquad
    \frac{ 2 - {\tau} + \abs \gamma } d < q_2\inv < 1.
\end{equation}
Then $\hat F, \hat A, \hat E$ are $\gamma$-times weakly differentiable and
\begin{align}  \label{eq:EAF_S}
\Bignorm{ \frac{ \hat A_\gamma }{ \hat A } }_{q_1}, \;
\Bignorm{ \frac{ \hat F_\gamma }{ \hat F } }_{q_1}, \;
\Bignorm{ \frac{ \hat E_\gamma }{ \hat A \hat F }}_{q_2}
\lesssim 1,
\end{align} with constants independent of $L$.
\end{lemma}

Given Lemma~\ref{lem:AFEbds_S}, it follows exactly as in the proof of Proposition~\ref{prop:f_SO} that
the function $\hat f = \hat E / (\hat A \hat F)$
is $d-1$ times weakly differentiable, and
\begin{align} \label{eq:f_alpha_S}
\norm {   \hat f_\alpha  }_{r} \lesssim 1
	\qquad (r\inv > \frac{ \abs \alpha }d)
\end{align}
for $\abs \alpha \le d-1$, with the constant independent of $L$.
We use this to prove Lemma~\ref{lemma:phi} first, and then we complete the proof of
Proposition~\ref{prop:so-nn}
by proving Lemma~\ref{lem:AFEbds_S}.

\begin{proof}[Proof of Lemma~\ref{lemma:phi} assuming Lemma~\ref{lem:AFEbds_S}]
For simplicity, we write $C = C_1$ and $\hat C = 1/\hat A$.
It suffices to consider small $\eps>0$.
By the product rule, and since $\sigma^{-2} \lesssim L^{-2}$,
it suffices to prove that, for $|\alpha_1|+|\alpha_2|=|\alpha| \le d-1$,
\begin{align}
\label{eq:Df}
     \|\hat f_{\alpha_1}\hat D_{\alpha_2}\|_1 &\lesssim L^{-(1-\eps)} , \\
\label{eq:CF}
    \|\hat C_{\alpha_1}\hat F_{\alpha_2}\|_1 &\lesssim L^{1+\eps} .
\end{align}
The bound \eqref{eq:Df} directly follows from H\"older's inequality, \eqref{eq:f_alpha_S}, and \eqref{eq:D_alpha}:
\begin{align}
\norm{ \hat f_{\alpha_1}\hat D_{\alpha_2} }_{1}
&\le
	\norm{ \hat f_{\alpha_1} }_{\frac d {\abs {\alpha_1} + \eps }}
	\norm{ \hat D_{\alpha_2} }_{\frac d { \abs{\alpha_2} +1 - \eps } }
\lesssim  L^{ \abs{\alpha_2} - (\abs{\alpha_2} + 1-\eps) } = L^{-(1-\eps)}.
\end{align}
To prove \eqref{eq:CF}, we first note that, by
a direct computation using the explicit formula
\begin{align} \label{eq:C}
\hat C_1(k) = \frac 1 { \hat A(k) } = \frac 1 { 1 - \hat \Dnn(k) }, \qquad
\Dnn(k) = d^{-1}\sum_{j=1}^d \cos k_j,
\end{align}
we have $\abs{ \hat C_{\alpha_1} (k) } \lesssim \abs k^{-(2 + \abs {\alpha_1} )}$ for all $\alpha_1$.
Consider first the case $\abs{\alpha_2}<1+\eps$.
Since $\sum_{x\in\Z^d} F(x) = 0$ and $1+\eps \le 2$, Taylor expansion at $k=0$ gives
\begin{align} \label{eq:F_alpha_S}
\abs{ \hat F_{\alpha_2 }(k) }
\le \sum_{x\in\Z^d} \abs{ F(x) } \abs{ \grad^{\alpha_2}(\cos (k\cdot x) - 1) }
\lesssim \sum_{x\in\Z^d} \abs{ F(x) } \abs k^{1+\eps - \abs{ \alpha_2 } } \abs x^{1+\eps}
\lesssim L^{1+\eps} \abs k^{1+\eps - \abs{\alpha_2} },
\end{align}
so that \begin{align}
\norm{  \hat C_{\alpha_1} \hat F_{\alpha_2} }_1
\lesssim  L^{1+\eps} \norm{ \abs k^{ -(2 + \abs{\alpha_1} + \abs{\alpha_2} - 1 - \eps )} }_1 \lesssim L^{1+\eps},
\end{align}
since $2 + \abs{ \alpha_1 } + \abs{ \alpha_2 } - 1 - \eps \le d-\eps < d$.
If instead
$\abs{ \alpha_2} \ge 1+\eps$, then
we use H\"older's inequality and the norm bound \eqref{eq:D_alpha}
 on $\hat F = 1 - \hat D$ to
complete the proof of \eqref{eq:CF} with
\begin{align}
\norm{  \hat C_{\alpha_1} \hat F_{\alpha_2 } }_1
\le  \norm{ \hat C_{\alpha_1} }_{\frac{d}{2 + \abs{ \alpha_1 } + \eps} }
	\norm{ \hat F_{\alpha_2} }_{\frac{d}{ \abs{ \alpha_2 } - 1-\eps} }
\lesssim  L^{  \abs{ \alpha_2 }    -( \abs{ \alpha_2 } - 1-\eps )  }
= L^{1+\eps}.
\end{align}
This concludes the proof.
\end{proof}

\begin{proof}[Proof of Lemma~\ref{lem:AFEbds_S}]
The claim on $\hat A_\gamma / \hat A$ follows from Taylor expansion and the explicit, $L$-independent formula \eqref{eq:C} (see \cite[Lemma~2.5]{LS24a}).
The claim on $\hat F_\gamma / \hat F$ follows from Lemma~\ref{lem:AFEbds_SO}, with $F = \delta - D$ and $\beta=0$.
We cannot immediately apply our previous bounds to $\hat E_\gamma / (\hat A \hat F)$, because now it mixes both $\Dnn$ and $D$, and in particular
we now have the two different infrared bounds
\begin{align}
\hat A(k) \gtrsim \abs k^2, \qquad \hat F(k) \gtrsim L^2 \abs k^2 \wedge 1.
\end{align}
We adapt the proof of Lemma~\ref{lem:AFEbds_SO} to bound $\hat E_\gamma / (\hat A \hat F)$, as follows.
By the infrared bounds, \begin{align} \label{eq:EAF_decomp_S}
\Bigabs{ \frac{ \hat E_\gamma }{ \hat A \hat F } (k) }
\lesssim L^{-2} \abs k^{-4} \abs{ \hat E_\gamma(k) } \1_{B_L} + \abs k^{-2} \abs{ \hat E_\gamma(k) } .
\end{align}
Recall that $E$ is given by
\begin{align} \label{eq:E_S}
E = A - \sigma^{-2} F = (\delta - \Dnn) - \sigma^{-2} ( \delta - D ),
\end{align}
so we have
$\sum_{x\in\Z^d} \abs x^2 \abs { E(x) } \le 1+1 = 2$.
Also, by the norm estimates \eqref{eq:D_alpha} of $\hat D_\gamma$, we have
\begin{align} \label{eq:E_gamma_S}
\norm{ \hat E_\gamma }_{r}
\lesssim 1 + L^{-2} ( 1 + L^{\abs \gamma - d/r})
\le  2 +  L^{\abs \gamma -2 - d/r}
	\qquad ( 0 \le r\inv < 1 ).
\end{align}

Let ${\tau} \in (0,2)$ and $q = q_2 \in ( \frac{ 2 - {\tau} + \abs \gamma } d ,1)$.
We start with the second term of \eqref{eq:EAF_decomp_S}.
Suppose first that $\abs \gamma < {\tau}$, so $\abs \gamma \in \{ 0, 1 \}$.
By symmetry and $\sum_{x\in\Z^d} E(x)=0$, it follows that $\hat E(k) = \sum_{x\in\Z^d} E(x) (\cos (k\cdot x)  - 1)$.
Taylor expansion of $\cos(k\cdot x) - 1$ or its derivative in $k$ gives
\begin{align}
\abs { \hat E_\gamma (k) }
\lesssim  \sum_{x\in\Z^d} \abs k^{2-\abs \gamma} \abs x^2 \abs{ E(x) }
\lesssim \abs k^{ 2 - \abs \gamma},
\end{align}
so that the $L^q$ norm of $ \abs k^{-2} \abs{ \hat E_\gamma(k) } \lesssim \abs k^{ - \abs \gamma}$ can be bounded independent of $L$ when $q \inv > \abs \gamma /d$.
This includes the desired $q$'s because ${\tau} \le 2$.
If $\abs \gamma \ge  {\tau}$, we observe that every $q$
with $\frac{ 2 - {\tau} + \abs \gamma } d < q\inv < 1$
can be written
in the form $q \inv = r\inv + p\inv$ for some
$r\inv \in ( \frac{ \abs \gamma - {\tau} } d, 1 )$ and $p\inv > 2/d$.
Then \eqref{eq:E_gamma_S} and H\"older's inequality gives
\begin{align}
\bignorm{ \abs k^{-2} \hat E_\gamma (k) }_q
\le \bignorm{ \abs k^{-2} }_p \norm{ \hat E_\gamma }_r
\lesssim 1 + L^{\abs \gamma - 2 - d/r}
\le 1 + L^{ {\tau} - 2 }
\le 2
.
\end{align}
Therefore, the $L^q$ norm of the second term on the right-hand side of \eqref{eq:EAF_decomp_S}
is bounded uniformly in $L$.

For the first term of \eqref{eq:EAF_decomp_S}, suppose first that $\abs \gamma < 2 + {\tau}$.
By \eqref{eq:E_S} and the decay \eqref{eq:D_decay} of $D$,
\begin{align} \label{eq:E_decay_S}
\abs{ E(x) }
\lesssim \frac{ 1 + L^{-2}(1+L^{4}) }{ \xvee^{d+4} }
\lesssim \frac{ L^{2} }{ \xvee^{d+4} }.
\end{align}
Lemma~\ref{lemma:Ek} with $\rho = 2$ then gives $\abs{ \hat E_\gamma (k) } \lesssim L^{2} \abs k^{2+{\tau} - \abs \gamma}$.
Since the $L^2$ cancels with $L^{-2}$, the $L^q(B_L)$ norm of the first term on
on the right-hand side of \eqref{eq:EAF_decomp_S} is exactly that of $\abs k^{2+{\tau}-\abs\gamma -4} = \abs k^{-(2-{\tau} + \abs \gamma)}$, which is of order $L^{2 - {\tau} + \abs \gamma - d/q} \le 1$ for $q\inv > ( 2 - {\tau} + \abs \gamma)/d$,
as desired.
For the remaining case $\abs \gamma \ge 2 + {\tau}$, H\"older's inequality and \eqref{eq:E_gamma_S} imply that the $L^q$ norm of the first term on the
right-hand side of \eqref{eq:EAF_decomp_S} is bounded by
\begin{align}
\frac 1 {L^2} \norm{ \hat E_\gamma }_r \bignorm{  \abs k ^{-4} \1_{B_L}}_p
&\lesssim \frac 1 {L^2} ( 1 + L^{\abs \gamma -2 - d/r} ) L^{4 - d/p}
\lesssim L^{\abs \gamma - d/q}
\end{align}
for $q\inv = r \inv + p \inv$, $r\in (1, \infty]$,
$ \frac { \abs \gamma - 2 - {\tau} } d < r \inv \le  \frac { \abs \gamma - 2 }  d$, and $p \inv > 4/d$. In particular, the bound holds for all $q\inv > ( 2 - {\tau} + \abs \gamma)/d$. The desired bound then follows from ${\tau} \le 2$.
This completes the proof.
\end{proof}

\section{Proof of Lemma~\ref{lemma:D_L}}
\label{app:D}

Lemma~\ref{lemma:D_L} makes the following assertion for $D$
defined by Definition~\ref{def:D}, for $L\ge L_0$
with $L_0$ sufficiently large depending
only on $d$ and $v$:
\begin{gather}
\label{eq:D_decay-B}
D(x) \lesssim \frac{  L^{a} } { \xvee^{d + a} }
\quad \text{for any $a>0$},
\\
\label{eq:AIR-B}
\hat A_\muSO(k) - \hat A_\muSO(0) \gtrsim  L^2 \abs k^2 \wedge 1
\quad \text{for all $k\in\T^d$, uniformly in $\muSO \in [\tfrac 12,1]$},
\\
\label{eq:D_alpha-B}
\norm {  \hat D_\alpha }_{q } \lesssim L^{\abs\alpha - d/q} 	
\quad \text{for each $\alpha$ and for all $0 \le q\inv < 1$}.
\end{gather}

\begin{proof}
For \eqref{eq:D_decay-B}, let $b >0$.
By definition, $D(x) \lesssim L^{-d}\1_{\|x\|_\infty \le L}$, so
$D(x) \lesssim L^{-d}(L/|x|)^b$, and the desired result follows by choosing $b=d+a$.

The infrared bound \eqref{eq:AIR-B} is
proved in \cite[Appendix~A]{HS02}.

It remains to prove \eqref{eq:D_alpha-B}.
For $q=\infty$, we simply observe that $|\hat D_\alpha(k)| \le \sum_{x\in\Z^d} |x^\alpha |D(x)
\le L^{|\alpha|}$.
For $q<\infty$, we divide the integral according to whether or not $\norm k_\infty \le 1/L$.
When $\norm k_\infty \le 1/L$, we use $|\hat D_\alpha(k)|
\le L^{|\alpha|}$.  Since the volume is
of order
$L^{-d}$, we obtain the desired upper bound $L^{|\alpha|-d/q}$.
When $\norm k_\infty > 1/L$, we apply \cite[(5.34)]{HS90a} (whose proof generalises
to any number of derivatives)
as in the proof of \cite[Lemma~5.7]{HS90a}, as follows.
The domain $\norm k_\infty > 1/L$ is the disjoint union over nonempty subsets
$S \subset \{1, \ldots, d\}$ of
\begin{equation}
	R_S = \{ k \in \R^d: 1/L < k_i \leq \pi \; \text{for} \;
	i \in S,\;\; |k_j| \leq 1/L \; \text{for} \; j \not\in S \}.
\end{equation}
By \cite[(5.34)]{HS90a}, for $q \in (1,\infty)$,
\begin{align}
	\int_{R_S} |\hat D_\alpha (k)|^q dk & \lesssim
	L^{q|\alpha|} \int_{R_S} \prod_{i \in S} |Lk_i|^{-q} dk
    \nnb &
    \lesssim
    L^{q(|\alpha|-|S|)} \Big( \int_{1/L}^\pi t^{-q} dt \Big)^{|S|}
    \Big( \int_0^{1/L}  1\, dt \Big)^{d-|S|}
	\nnb
	&\lesssim
    L^{q(|\alpha|-|S|)} L^{(q-1)|S|} L^{|S|-d} = L^{q|\alpha| -d}.
\end{align}
The desired result \eqref{eq:D_alpha-B} then follows by summing over $S$.
\end{proof}

\end{appendix}

\section*{Acknowledgement}
The work of both authors was supported in part by NSERC of Canada.


\end{document}